\documentclass[reqno]{amsart}
\usepackage{amssymb}
\usepackage{hyperref}

\begin{document}
\title[ ]
{Spectral inclusion  for $C_0$-Semigroups Drazin invertible and Quasi-Fredholm operators}

\author[ A. TAJMOUATI, M. AMOUCH,  M. R. F. Z. ALHOMIDI  ]
{  A. TAJMOUATI, M. AMOUCH,  M. R. F.Z. ALHOMIDI }  

\address{A. TAJMOUATI and M. R. F. ALHOMIDI ZAKARIYA  \newline
 Sidi Mohamed Ben Abdellah
 Univeristy
 Faculty of Sciences Dhar Al Mahraz Fez, Morocco.}
\email{abdelaziztajmouati@yahoo.fr }
\email{ zakariya1978@yahoo.com}

\address{M. AMOUCH \newline
Department of Mathematics
University Chouaib Doukkali,
Faculty of Sciences, Eljadida.
24000, Eljadida, Morocco.}
\email{mohamed.amouch@gmail.com}



\newtheorem{theorem}{Theorem}[section]
\newtheorem{definition}{Definition}[section]
\newtheorem{remark}{Remark}
\newtheorem{lemma}{Lemma}[section]
\newtheorem{proposition}{Proposition}[section]
\newtheorem{corollary}{Corollairy}[section]
\newtheorem{example}{Example}
\newcommand{\pr} {{\bf   Proof: \hspace{0.3cm}}}

\maketitle

\begin{abstract}
Let $(T(t))_{t\geq0}$ be a $C_0$ semigroups and
$A$ be its infinitesimal generator.
In this work, we  prove that the spectral inclusion for $(T(t))_{t\geq0}$ remains true for the Drazin
invertible and  Quasi-Fredholm spectra. Also, we will give conditions under which
facts $A$ is quasi-Fredholm,  $A$ is Drazin invertible and $A$ is B-Fredholm
are equivalent.
\end{abstract}

\section{Introduction and preliminaries.}
Let $X$ a Banach space and $\mathcal{B}(X)$ the Banach algebra of all bounded linear
operators on $X.$
for $T\in \mathcal{B}(X),$ by $T^*$, $N(T)$,  $R(T)$,
$ R^{\infty}(T)=\bigcap_{n\geq0}R(T^n)$, $N^{\infty}(T)=\bigcup_{n\geq0}N(T^n)$,  $\rho(T)$
and  $\sigma(T)$, we denote,
respectively the adjoint,  the null space, the range, the hyper-range,
the hyper-kernel, the resolvent set and  the spectrum of $T$.

Let  $T \in B(X)$, the ascent  $a(T)$ and the descent $d(T)$
of $T$  are defined in \cite{Tl} by $a(T) = \inf\{n\in\mathbb{N} : N(T^n) =
N(T^{n+1})\}$ and $d(T) = \inf\{n \in\mathbb{N} : R(T^n) = R(T^n+1)\},$
respectively.

The operator $T$ is said to be Drazin invertible if $d(T) < \infty$ and $a(T) < \infty.$
It is well known that $T$ is Drazin invertible if and only if $T = T_1 \oplus T_2$ where $T_1$ is
invertible and $T_2$ is nilpotent, see \cite[ Corollary 2.2]{L}.
This is equivalent to the fact that,
there exists an  integer $n$  such that the space $R(T^n)$ is closed and the restriction of $T$ of
$R(T^n)$ viewed as a map from $R(T^n)$ into $R(T^n)$ is invertible, see \cite[Theorem 2.5]{Abc}.
The Drazin spectrum of $T$ is defined by
$$\sigma_{D}(T ) = \{\lambda \in\mathbb{C}  : \lambda -T \mbox{  is not  Drazin invertible}\}.$$
Similarly, from \cite{Hbm}, $T$ is left (respect right) generalized Drazin invertible if and only
if $T = T_1 \oplus T_2$ such that $T_1$ is left (respect right) invertible and $T_2$ is quasi-nilpotent.

Recall that $T$ is said to be semi-regular or Kato  operator,
if $R(T)$ is closed and $N(T) \subset R^{\infty}(T),$ see for example \cite{AI}.
In addition, $T$ is said to be pseudo-Fredholm operator if there exist two
closed T-invariant subspaces $M$ and $N$ such that
$X=M\oplus N$  and  $T=T_{\mid M } \oplus T_{\mid N}$ with $T_{\mid M}$ is semi-regular and $T_{\mid N}$ is
nilpotent. This is equivalent to the fact that
there exists an integer  $n$ such that  $R(T^n)$ is closed and the restriction of $T$ of
$R(T^n)$ viewed as a map from $R(T^n)$ into $R(T^n)$ is semi-regular.
The quasi-Fredholm spectrum is defined by $$\sigma_{QF}(T ) = \{\lambda \in\mathbb{C}  : \lambda -T \mbox{  is  not  quasi-Fredholm}\},$$
see \cite{Hbm, Mb1, Mb2} for more information.

Similarly, from \cite{BE} $T$ is said to be a B-Fredholm operator, if there exists an integer $n$
such that the space $R(T^{n})$  is closed and the restriction of $T$ of $ R(T^{n})$ viewed as a map from $R(T^{n})$
into $ R(T^{n})$ is Fredholm.

Let $T\in B(X)$ and $x\in X,$ the local resolvent of $T$ at $x$
noted $\rho_{T}(x)$ is defined as the union  of all  open subset $U$ of
$\mathbb{C}$ for which there is an analytic function
$f: U\rightarrow X$ such that the equation $(T-\mu I)f(\mu)=x$ holds for
all $ \mu \in U$.
The local spectrum $\sigma_T(x)$ of $T$
at $x$ is defined by $\sigma_T(x)=\mathbb{C}\setminus \rho_T(x).$
Evidently  $\rho_T(x)$ is an open subset of $\mathbb{C}$ and $\sigma_T(x)$ is closed.
If $f(z)=\displaystyle{\sum_{i=0}^{\infty}}x_i(z-\mu)^i$  ( in  a neighborhood of $\mu$),
then $\mu\in\rho_{T}(x)$ if and only if there exists a sequence $(x_i)_{i\geq0}\subseteq X$, $x_0=x$,  $(T-\mu)x_{i+1}=x_i$, and $\displaystyle{\sup_i}||x_i||^{\frac{1}{i}}<\infty$, see \cite{lau}.\\

Let $\mathcal{T} = (T (t))_{t\geq0}$ be a strongly continuous semigroup
($C_0$ semigroup in short) with infinitesimal generator $A$ on $X$.
We will denote the type (growth bound) of $\mathcal{T}$ by $\omega_{0}$:
$$\omega_{0}:=\lim_{t\rightarrow\infty}\frac{\ln\|T't)\|}{t}$$
$$=\inf \big \{\omega \in \mathbb{R}~~: ~~~~~~\mbox{there existe M such that}\,\, \|T(t)\|\leq Me^{\omega t },t\geq0 \big \},$$
see \cite{ Cha, En, P} for more information. Also, in \cite{ Cha, En, P} the authors showed that
$$e^{t\nu (A)}\subseteq  \nu(T(t))\subseteq e^{t\nu (A)}\bigcup \{0\}$$
where $\nu(.)\in \{\sigma_p(.),\sigma_{ap}(.) ,\sigma_r(.)\}$ is the point spectrum, approximative spectrum or residual spectrum.\\
The semigroup $T(t)$ is called differentiable for $t> t_0$ if for every $x \in  X, ~t \rightarrow T(t)x$ is differentiable for $t > t_0$. $T(t)$ is called differentiable if it is differentiable for $t > 0.$
If $B(\lambda,t)x = \int_{0}^{t}e^{\lambda(t-s)}T(s)xds,$ then
$B(\lambda,t)x$ is differentiable in $t$ with
$B^{'}(\lambda,t)x = T(t)x + \lambda B(\lambda,t)x$ and
$B_{\lambda}^{'}(t)$ is a bounded linear operator in $X,$ see \cite{En} and \cite{P}.

Spectral inclusions for various reduced spectra of a $C_0$ semigroup was studied by authors in \cite{En} and \cite{Cha}
for point spectrum, approximative spectrum and residual spectrum. Also, the
spectral equality for a $C_0$ semigroup was studied by authors in \cite{Aaz}
for semi-regular, essentially semi-regular and semi-Fredholm spectrum, respectively.
In this work, we will continue in this direction, we will prove that the spectral inclusion
for Drazin and quasi Fredholm spectra.
Also, we will show that if $(T(t))_{t\geq0}$ is a $C_0$ semigroup with infinitesimal generator
$A$ such that the equality $\lim_{t\rightarrow \infty}\frac{1}{t^n}\|T(t)\|=0$ holds for some $n\in\mathbb{N}$,
then the infinitesimal generator $A$ is quasi-Fredholm if and only if it is Drazin invertible if and only if it
is B-Fredholm.

\section{ Main results }

\noindent We start by given the following two lemmas which are proved in \cite{En}. They will be used to prove our main result.

\begin{lemma}\cite{En}\label{l1}
Let $(A,D(A))$ be the infinitesimal generator of a strongly continuous semigroup
$(T(t))_{t\geq0}$ and $B(\lambda,t)=\int_{0}^{t}e^{\lambda(t-s)}T(s)xds$ is  a bounded operator  from $X$ to $D(A).$
Then, for every $\lambda \in\mathbb{C}$, $t > 0$ and $n\in \mathbb{N}$,
the following statements hold:
\begin{enumerate}
   \item $(a):(e^{\lambda t}- T(t))^{n}(x)=(\lambda-A)^{n}B(\lambda,t)^{n}(x) ,~~~~~~\lambda \in\mathbb{C},x\in X;$\\
   $(b):(e^{\lambda t}- T(t))^{n}(x)=B(\lambda,t)^{n}(\lambda-A)^{n}(x)  ,~~~~~~\lambda \in\mathbb{C},x\in D(A).$
                                                     \item $ R(e^{\lambda t} -T(t))^{n}\subseteq R(\lambda-A)^{n} .$
                                                     \item $ N(\lambda-A)^{n}\subseteq N(e^{\lambda t} -T(t))^{n}.$

                                                   \end{enumerate}
\end{lemma}

\begin{lemma}\cite{En}\label{l2}
Let $(A,D(A))$ be the infinitesimal generator of a strongly continuous semigroup
$(T(t))_{t\geq0}$ on $X$ and assume that the restricted semigroup $(T(t)_{|})_{t\geq0}$ is strongly
continuous on some $(T(t))_{t\geq0}$-invariant Banach space $Y\hookrightarrow X$. Then the infinitesimal generator
of $(T(t)|)_{t\geq0}$ is the part $(A_{|},D(A_{|}))$ of $A$ in $Y $.
\end{lemma}

\noindent The following two lemmas which are proved in \cite{Aaz} will be used in the sequel.

\begin{lemma}\cite[Lemma 3.1]{Aaz}\label{l3}
Let $A$ the infinitesimal generator of $C_0$-semigroup $T(t)_{t\geq0}$ and $B(\lambda;t)=\int_{0}^{t}e^{\lambda(t-s)}T(s)ds $
is a  linear bounded operator on $X$.
Then there exist $C$ and $D $ tow operator such that $(\lambda -A),B(\lambda,t),C,D$ are mutually commuting operators, for all $x\in D(A)$
and $C(\lambda -A)+DB(\lambda,t)=I,$ for $t>0$.
\end{lemma}

\begin{lemma}\cite[Lemma 3.2]{Aaz}\label{l4}
Let $(\lambda-A),B(\lambda,t),~~C,~~D$ be mutually commuting operators in $D(A)$ such that
$C(\lambda -A)+DB(\lambda,t)=I$, $t>0.$ Then we have:
\begin{enumerate}
 \item For every  positive integer $n$ there are $C_n,D_n \in D(A)$ such that\\
  $(\lambda-A)^{n},B^{n}(\lambda,t),C_{n},D_{n}$ are mutually
commuting and $$(\lambda-A)^{n}C_{n} + B^{n}(\lambda,t)D_n = I.$$
\item For every positive integer $n,$
$R(e^{\lambda t}-T(t))^{n} = R(\lambda-A)^{n} \bigcap R(B^{n}(\lambda,t))$ and
                                        $N((e^{\lambda t}-T(t))^{n}) = N((\lambda-A)^{n}) + N(B^{n}(\lambda,t)).$
Further $R^{\infty}(e^{\lambda t}-T(t)) = R^{\infty}(\lambda-A) \bigcap R^{\infty}(B(\lambda,t))$
and $N^{\infty}(e^{\lambda t}-T(t)) = N^{\infty}(\lambda-A) + N^{\infty}(B(\lambda,t)).$
\item $N^{\infty}(\lambda-A) \subset R^{\infty}(B(\lambda,t)) $and $N^{\infty}(B(\lambda,t))\subset  R^{\infty}(\lambda-A).$
                                     \end{enumerate}
\end{lemma}

\noindent Now, we give some  spectral results for differentiable $C_0$ semigroup.

\begin{lemma}\label{l5}
Let $(A,D(A))$ be the infinitesimal generator of a strongly continuous semigroup
$(T(t))_{t\geq0}$, $B(\lambda,t)=\int_{0}^{t}e^{\lambda(t-s)}T(s)xds$ is  a bounded operator  from $X$ to $D(A).$
If $T(t)$ is differentiable for $t > t_0,$
then for every $\lambda \in\mathbb{C}$, $t > t_0$ and $n\in \mathbb{N}$,
 the following statements hold:
\begin{enumerate}
   \item $(a):(\lambda e^{\lambda t}- AT(t))^{n}(x)=(\lambda-A)^{n}B(\lambda,t)^{'n}(x) ,~~~~~~\lambda \in\mathbb{C},x\in X;$\\
   $(b):(\lambda e^{\lambda t}- T(t))^{n}(x)=B(\lambda,t)^{'n}(\lambda-A)^{n}(x)  ,~~~~~~\lambda \in\mathbb{C},x\in D(A).$
                                                     \item $ R(\lambda e^{\lambda t} -AT(t))^{n}\subseteq R(\lambda-A)^{n}. $
                                                     \item $ N(\lambda-A)^{n}\subseteq N(\lambda e^{\lambda t} -AT(t))^{n}.$
                                          \end{enumerate}
\end{lemma}

\begin{proof}
Assuming now that $t > t_0$ and
differentiating $(a)$ and $(b)$ in $(1)$ of Lemma \ref{l1} with respect to $t$ and $n=1$, we obtain
\begin{eqnarray*}
\lambda e^{\lambda t}x - AT(t)x &=& (\lambda I - A)B^{'}(\lambda,t)x ~~~~\mbox{for every}~~ x \in X;\\
\lambda e^{\lambda t}x - AT(t)x &=&B^{'}(\lambda,t)x(\lambda I - A) ~~~~\mbox{for every}~~ x \in D(A) .
\end{eqnarray*}
This gives $(a)$ and $(b)$ for $n=1.$
By induction, we obtain $(a)$ and $(b)$ for all $n\in \mathbb{N}.$ The rest of Lemma follows from $(1).$
\end{proof}

\begin{proposition}
Let $T( t)_{t>0}$ be a $C_0$ semigroup and let $A$ be its infinitesimal
generator. If $T(t)$ is differentiable for $t > t_0$ and $\lambda \in \sigma_{A}(x)$ for $x\in X,$ then
$$\lambda e^{\lambda t} \in \sigma_{AT(t)}(x).$$
\end{proposition}

\begin{proof}
Let $t>t_{0}$ be fixed and suppose that $\lambda e^{\lambda t}\notin \sigma_{AT(t_0)}(x)$, then there exist
a sequence $(x_i)_{i\in \mathbb{N}}$ of $X$
such that $x_0=x$, $(\lambda e^{\lambda t}-AT(t))x_i=x_{i-1}$
and $\displaystyle{\sup_i}\|x_i\|^{\frac{1}{i}}<\infty.$\\
We put $y_i=B^{'i}(\lambda,t)x_i$, as $B^{'}(\lambda,t)$ is a bounded linear operator in $X$, we have $y_0=x_0=x,~~y_0\in D(A)$.\\
\begin{eqnarray*}
  (\lambda-A)y_i &=& (\lambda-A)B^{'}(\lambda,t)x_i B^{'(i-1)}(\lambda,t)x_i \\
   &=& (\lambda e^{\lambda t}-AT(t))B^{'(i-1)}(\lambda,t)x_i \\
   &=& B^{'(i-1)}(\lambda,t)(\lambda e^{\lambda t}-AT(t))x_i \\
   &=& B^{'(i-1)}(\lambda,t)x_{i-1}\\
   &=& y_{i-1}
\end{eqnarray*}

Therefor $(\lambda-A)y_i=y_{i-1}.$
On the other hand $$\|y_i\|=\| B^{'i}(\lambda,t)x_i\|<\|B^{'i}(\lambda,t)\| \|x_i\|<M^i\|x_i\|,$$
then $$\displaystyle{\sup_i}\|y_i\|^{\frac{1}{i}}< \displaystyle{\sup_i} M\|x_i\|^{\frac{1}{i}}<\infty.$$
So that $\lambda \notin \sigma_A(x)$
\end{proof}

\noindent Denote by $\sigma_{su}(A)$ the subjectivity spectrum of $A$. It is known that
$\displaystyle{\bigcup_{x\in X}}\sigma_{A}(x)= \sigma_{su}(A)$ for a closed operator $A.$
Hence the following corollary holds.

\begin{corollary}
Let $T( t)_{t>0}$ be a $C_0$ semigroup and let $A$ be its infinitesimal
generator. If $T(t)$ is differentiable for $t > t_0$ and $\lambda \in \sigma_{su}(A),$ then
$\lambda e^{\lambda t} \in \sigma_{su}(AT(t)).$
\end{corollary}

\begin{proposition}
Let $T( t)_{t>0}$ be a $C_0$ semigroup and let $A$ be its infinitesimal
generator. If $T(t)$ is differentiable for $t > t_0$ and $\lambda \in \sigma_{ap}(A)$,  then
$\lambda e^{\lambda t} \in \sigma_{ap}(AT(t)).$
\end{proposition}
\begin{proof}
For $t > t_0$, since $\lambda e^{\lambda t}-AT(t)=(\lambda -A)B^{'}(\lambda,t)$
and $B^{'}(\lambda,t)$ is bounded linear operator, then $\lambda e^{\lambda t}-AT(t)$
is a bounded linear operator. It is easy to check that for $x\in D(A)$, $B^{'}(\lambda,t)Ax=AB^{'}(\lambda,t)x.$
  If $\lambda \in \sigma_{ap}(A)$, then there exists sequence $(x_n)_{n\in\mathbb{N}}\in D(A)$  satisfying $||x_n||=1$ and
  $||(\lambda -A)x_n||\rightarrow 0.$
From $(1)$ of lemma \ref{l5}, we obtain the result.
\end{proof}

\noindent By an outline of the proof of \cite[Theorem 2.1]{Te}, we obtain the following result.
\begin{proposition}
Let $T( t)_{t>0}$ be a $C_0$ semigroup and let $A$ be its infinitesimal
generator. If $T(t)$ is differentiable for $t > t_0$ and $\lambda \in \nu(A)$,  then
$\lambda e^{\lambda t} \in \nu(AT(t)).$
where $\nu(.) \in \{\sigma_{\gamma}(.),\sigma_{\pi}(.),\sigma_{\gamma e}(.)\}$ and $\sigma_{\gamma}(.),\sigma_{\pi}(.),\sigma_{\gamma e}(.)$
denote the regular spectrum, essential regular spectrum and left essential spectrum.
\end{proposition}

\noindent In the next theorem we will prove  that the spectral inclusion of
$C_0$semi-groups remains true for the Drazin invertible and quasi-Fredholm spectra.

\begin{theorem}\label{td}
Let $(T(t))_{t\geq0}$ a $C_0-$semigroup, with infinitesimal generator $A.$ Then
$$ e^{t\sigma_{D}(A)}\subseteq \sigma_{D}(T(t)).$$
\end{theorem}

\begin{proof}
Let $t_{0}>0$ be fixed and suppose that $(e^{\lambda t_0} -T(t_0))$ is Drazin invertible
for some $\lambda\in \mathbb{C}\setminus \{0\}$.
Then $M:= R(e^{\lambda t_0} -T(t_0))^n$ is closed and the restricted semigroup $(e^{\lambda t_0} -T(t_0)_{\mid M})$ is invertible.
We show that $(\lambda-A)$  is Drazin invertible. To this end, in the first we show that $R(\lambda-A)^{n}$ is closed.
Let $x\in\overline{R((\lambda-A)^n))}$, that is, there exist $u_{k}\in D(A^n),$
$ k=1,2,...,$ such that $(\lambda-A)^{n}u_{k}\rightarrow x,$ hence
$$(e^{\lambda t_0} -T(t_0))^{n}u_{k} := B(\lambda,t_0)^{n}(\lambda-A)^{n}u_{k}$$ by Lemma \ref{l1}.
Also,
$$B(\lambda,t_0)^{n}(\lambda-A)^{n}u_{k}~~\rightarrow B(\lambda,t_0)^{n}x.$$
Hence $$(e^{\lambda t_0} -T(t_0))^{n}u_{k}~~\rightarrow B(\lambda,t_0)^{n}x.$$
Since $M:= R(e^{\lambda t_0} -T(t_0))^n$ is closed, then
$$B(\lambda,t_{0})^{n}x \in R(e^{\lambda t_0} -T(t_0))^n. $$
Hence there exists $u\in D(A^n)$ such that
$$B(\lambda,t_0)^{n}x= (e^{\lambda t_0} -T(t_0))^n u.$$
In the other hand, From Lemma \ref{l1}, we have that
$$(e^{\lambda t_0} -T(t_0))^n u = B(\lambda,t_0)^{n}(\lambda -A)u,$$
hence $$B(\lambda,t_0)^{n}x= B(\lambda,t_0)^{n}(\lambda -A)u.$$
This implies that, $$x- (\lambda-A)^{n}u \in N(B(\lambda,t_{0}))\subseteq R(\lambda-A)^n.$$
So $x\in R(\lambda-A)^n$, and hence $R(\lambda-A)^n$ is closed.\\
Now, let us to show that $(\lambda-A_{ \mid R(\lambda-A)^n})$ is invertible. For this,
as  $(e^{\lambda t_0} -T(t_0)_{\mid M})$ is invertible,
then $(e^{\lambda t_0} -T(t_0)_{\mid M})$ is bounded below and $R(e^{\lambda t_0} -T(t_0)_{\mid M})=R(e^{\lambda t_0} -T(t_0))^{n+1}$ is onto.
We show that $(\lambda-A_{\mid R(\lambda -A)^{n} \bigcap D(A)})$ is bounded below.
Since $(e^{\lambda t_0} -T(t_0)_{\mid M})$ is bounded below, then
$(e^{\lambda t_0} -T(t_0)_{\mid M})$ is injective and $R(e^{\lambda t_0} -T(t_0)_{\mid M})$ is closed.
We show that $(\lambda-A_{\mid R(\lambda -A)^{n} \bigcap D(A)})$ is injective and $(\lambda-A_{\mid R(\lambda -A)^{n} \bigcap D(A)})$
is closed.
For all $x\in D(A)$ we have
$$\{0\}=N(e^{\lambda t_0} -T(t_0)_{\mid M \bigcap D(A)})=
N(\lambda-A_{\mid R(\lambda -A)^{n} \bigcap D(A)})+N(B(\lambda,t_0))\bigcap R(B^{n}(\lambda,t_0)),$$
then $N(\lambda-A_{\mid R(\lambda -A)^{n} \bigcap D(A)})=\{0\},$
therefore $(\lambda-A_{\mid R(\lambda -A)^{n} \bigcap D(A)})$ is injective. As $R(e^{\lambda t_0} -T(t_0))^{n+1}$ is closed, then
$(\lambda-A_{\mid R(\lambda -A)^{n} \bigcap D(A)})=R(\lambda-A)^{n+1}$ is also closed.
On the other hand, as $R(e^{\lambda t_0} -T(t_0)_{\mid M})$ is onto, we can easily verify that
$R(\lambda-A_{\mid R(\lambda -A)^{n} \bigcap D(A)})$ is onto.
In fact one can verify that $$R(\lambda-A_{\mid R(\lambda -A)^{n} \bigcap D(A)})=R(\lambda-A)^{n+1}=R(\lambda-A)^{n}.$$
We have $R(\lambda-A)^{n+1}\subset R(\lambda-A)^{n}$ and if $y\in R(\lambda-A)^{n},$
then there exist $x\in D(A^{n})$ such that
$y= R(\lambda-A)^{n}x$ and by $(1)$ of Lemma \ref{l3}, we have
$$ (\lambda -A)^{n}x=(\lambda-A)^{n}C_{n}(\lambda-A)^{n}x+D_{n}B^{n}(\lambda,t)(\lambda-A)^{n}x$$
 $$=(\lambda-A)^{n+1}C_{n}(\lambda-A)^{n-1}x+D^{n}(e^{\lambda t_0}-T(t_0))^{n}x$$
 and as $R(e^{\lambda t_0}-T(t_0)\mid_{(e^{\lambda t_0}-T(t_0))^{n}})=R(e^{\lambda t_0}-T(t_0))^{n+1}$ is onto,
 then there exist $x^{'}\in X$   such that $(e^{\lambda t_0}-T(t_0))^{n}x=(e^{\lambda t_0}-T(t_0))^{n+1}x^{'}=(\lambda-A)^{n+1}B^{n+1}(\lambda,t)x^{'},$
 therefore $y\in R(\lambda-A)^{n}$ and then  $R(\lambda-A_{\mid R(\lambda -A)^{n} \bigcap D(A)})$ is onto.\\
 Finally, $(\lambda-A_{\mid R(\lambda -A)^{n} \bigcap D(A)})$ is Drazin invertible.\\

\end{proof}
\begin{corollary}
For the infinitesimal generator $A$ of a strongly continuous semigroup $(T(t))_{t\geq0}$  one  has the spectral inclusion
$$ e^{t\sigma_{\nu}(A)}\subseteq \sigma_{\nu}(T(t))$$
where $\sigma_{\nu}(.)$ is  the left Drazin and right Drazin spectra.
\end{corollary}

\noindent The following example shows that the inclusion in Theorem \ref{td} is strict.

\begin{example}
Let $X$ be the Banach space of continuous functions on $[0, 1]$
which are equal to zero at $x = 1$ with the supremum norm. Define
\begin{equation*}(T(t)f)(x):=\left\{
\begin{array}{lll}
f(x+t)   & \hbox{} if x+t\leq1; \\
               0    & \hbox{} if x+t>1.
\end{array}
\right.
\end{equation*}
$T(t)$ is obviously a $C_0$ semigroup on $X.$. Its infinitesimal
generator $A$ is given on
$$D(A) = \{f:f\in C^{1}([0, 1])\cup X,f^{'} \in X\}$$
by
$$Af=f^{'} \mbox{ for } f\in D(A).$$
One checks easily that for every $\lambda\in\mathbb{C} $ and $g \in X$ the equation $\lambda f - f^{'} = g$
has a unique solution $f \in X$ given by
$$f(t) = \int_{t}^{1}e^{\lambda(t-S)}g(s) ds.$$
Therefore $\sigma(A) = \emptyset$ see \cite{P}, hence $\sigma_{D}(A)=\emptyset.$\\
On the other hand, $T(t)$ is a
bounded linear operator for every $t \geq 0.$ Suppose that  $f\in X$ is not periodic,
then we have $0\in \sigma_{D}(T(t)).$ Indeed,
if $0\notin \sigma_{D}(T(t))$, then there exist a finite integer  $n\in \mathbb{N}$ such that $(T^{n}(t)f)(x)=(T^{n+1}(t)f)(x)$, this implies that
$f(x+nt)=f(x+nt+t)$, therefore $f$ is periodic, this is a contradiction. Hence the inclusion
$$ e^{t\sigma_{D}(A)}\subseteq \sigma_{D}(T(t)).$$ is strict.
\end{example}
\begin{theorem}
For the infinitesimal generator $A$ of a strongly continuous semigroup $(T(t))_{t\geq0}, $ we have the following inclusion:
$$ e^{t\sigma_{QF}(A)}\subseteq \sigma_{QF}(T(t)).$$
\end{theorem}
\begin{proof}
Let $t_{0}>0$ be fixed and suppose that $(e^{\lambda t_0} -T(t_0))$ is quasi-Fredholm,
for some $\lambda\in \mathbb{C}\setminus \{0\}$.
Then $M:= R(e^{\lambda t_0} -T(t_0))^n$ is closed and the restricted semigroup $(e^{\lambda t_0} -T(t_0)_{\mid M})$ is semi-regular.\\
We show that $(\lambda-A)$  is quasi-Fredholm, to this end we show that $R(\lambda-A)^{n}$ is closed and   $(\lambda-A_{|R(\lambda -A)^{n}\bigcap D(A)})$ is semi-regular. That is, ($R(\lambda-A_{|R(\lambda -A)^{n}\bigcap D(A)})$ is closed and
$N(\lambda-A_{|R(\lambda -A)^{n}\bigcap D(A)})\subseteq R^{\infty}(\lambda-A_{|R(\lambda -A)^{n}\bigcap D(A)})$).\\
As $M:= R(e^{\lambda t_0} -T(t_0))^n$ is closed, then by the same argument as in the proof of Theorem \ref{td}, we conclude that
$R(\lambda-A)^{n}$ is closed.
Since $(e^{\lambda t_0} -T(t_0)_{\mid M})$ is semi-regular, then $R(e^{\lambda t_0} -T(t_0)_{\mid M})=R(e^{\lambda t_0} -T(t_0))^{n+1}$
is closed, this implies that $R(\lambda-A_{|R(\lambda -A)^{n}\bigcap D(A)})=R(\lambda-A)^{n+1}$ is closed.
On the other hand, by Lemma \ref{l1} we have that
$$N(\lambda-A_{|R(\lambda -A)^{n}\bigcap D(A)})\subseteq N(e^{\lambda t_0} -T(t_0)_{\mid M\bigcap D(A)}\subseteq N(e^{\lambda t_0} -T(t_0)_{\mid M}).$$
Since $(e^{\lambda t_0} -T(t_0)_{\mid M})$ is semi-regular, then
$ N(e^{\lambda t_0} -T(t_0)_{\mid M}) \subseteq R^{\infty}(e^{\lambda t} -T(t_0)_{\mid M})=R^{\infty}(\lambda-A_{|R(\lambda -A)\bigcap D(A)})\bigcap R^{\infty}B(\lambda,t_0)\subset R^{\infty}(\lambda-A_{|R(\lambda -A)^{n}\bigcap D(A)}),$
hence $(\lambda-A_{|R(\lambda -A)^{n}\bigcap D(A)})$ is semi-regular.
Finally, we conclude that $(\lambda-A)$ is quasi-Fredholm.
\end{proof}

\noindent
Let $A$ be the infinitesimal generator of a $C_0$ semigroup $(T(t))_{t\geq0}.$
In the following,  we will give condition on $(T(t))_{t\geq0}$ under which  facts $A$ is Drazin invertible,
$A$ is $B-$Fredholm and $A$ is $Q-$Fredholm  are equivalent.

\begin{theorem}
Let $A$ be the infinitesimal generator of a $C_0$ semigroup $(T(t))_{t\geq0}$.\\ If  $\lim_{t\rightarrow \infty}\frac{1}{t^n}\|T(t)\|=0$, for some $n\in\mathbb{N}$,
the following assertions are equivalent:
\begin{enumerate}
  \item $A$ is quasi-Fredholm;
  \item $A$ is Drazin invertible;
  \item $A$ is B-Fredholm.
\end{enumerate}
\end{theorem}
\begin{proof}
$(1)\Rightarrow (2):$

Since $A$ is quasi-Fredholm, then $R(A^n)$ closed and $A_{\mid(R(A^n)\bigcap D(A))}$ is semi-regular.
Let $y\in N(A_{\mid(R(A^n)\bigcap D(A))}),$ then there exists
$x \in (R(A^n)\bigcap D(A^{n}))$ such that $y=A^{n}x$.
We integer by parts in the following formula:
$$T(t)x-x =\int_{0}^{t}T(s)Axds ,$$ we obtain that
$$T(t)x=x+tA+\frac{t^2}{2!}A^{2}+\int_{0}^{t}\frac{(t-s)^{2}}{2!}T(s)A^{3}xds.$$
We repeat this operation for n times, we obtain that
$$T(t)x=\sum_{k=0}^{n-1}\frac{t^k}{k!}A^{k}x+\int_{0}^{t}\frac{(t-s)^{n-1}}{(n-1)!}T(s)A^{n}xds.$$
Hence,$$T(t)x=\sum_{k=0}^{n-1}\frac{t^k}{k!}A^{k}x+y\int_{0}^{t}\frac{(t-s)^{n-1}}{(n-1)!}ds$$
$$=\sum_{k=0}^{n-1}\frac{t^k}{k!}A^{k}x+\frac{t^n}{n!}y.$$
As $\displaystyle{\lim_{t\rightarrow \infty}} \frac{1}{t^n}\|T(t)\|=0$, then $y=0$, this implies that
$$N(A_{\mid(R(A^n)\bigcap D(A))})=\{0\}.$$
On the other hand, let  $(T(t)^{'})_{t\geq0}$ with infinitesimal generator $A^{'}$  the adjoint semigroup of $(T(t))_{t\geq0}$.
Since $A_{\mid(R(A^n)\bigcap D(A))}$ is semi-regular, then $A^{'}_{\mid(R(A^{'n})\bigcap D(A^{'}))}$ is also semi-regular,
see \cite[Proposition 1.6]{Mb}.
Using the following formula
$$T(t)^{'}x^{'}-x^{'} =weak^{*}\int_{0}^{t}T(s)^{'}A^{'}x^{'}ds,~~\forall x^{'}\in (R(A^{'n})\bigcap D(A^{'})),~~\forall t\geq0 $$
which is proved in \cite[Proposition 1.2.2]{Vn} and by the same argument as above, we get that
$N(A^{'}_{\mid(R(A^{'n})\bigcap D(A^{'}))})=\{0\}$.
This is equivalent to the fact that $$\overline{R(A_{\mid(R(A^n)\bigcap D(A))})}= (R(A^n)\bigcap D(A)).$$
Hence $R(A_{\mid(R(A^n)\bigcap D(A))})= (R(A^n)\bigcap D(A)),$ since
$(R(A^n)\bigcap D(A))$ is closed.
From this it follows that  $A_{\mid(R(A^n)\bigcap D(A))}$ is surjective and hence it is invertible.
Finally, $A$ is Drazin invertible.\\
$(2)\Rightarrow (1)$: is clear.\\
$(1),$ $(2)$ and  $(3)$ are equivalent, since the class of Drazin invertible operator is  a subclass of B-Fredholm operator and the class of B-Fredholm operator  is a subclass of Quasi-Fredholm operator.

\end{proof}


\end{document}